\documentclass[a4paper,oneside,11pt]{article}

\usepackage{amsmath,amsfonts,amscd,amssymb}
\usepackage{longtable,geometry}
\usepackage[english]{babel}
\usepackage[utf8]{inputenc}
\usepackage[active]{srcltx}
\usepackage[T1]{fontenc}
\usepackage{graphicx}
\usepackage{pstricks}
\usepackage{bbm}
\usepackage{mathtools}
\usepackage{enumerate}
\usepackage{MnSymbol}
\usepackage{stmaryrd}
\usepackage{nicefrac}
\usepackage{calrsfs}
\usepackage{enumitem}
 \usepackage{rotating}
\usepackage{xcolor}
\usepackage{framed}

\colorlet{shadecolor}{blue!15}

\usepackage{amsthm}
\theoremstyle{plain}
\newtheorem{theorem}{Theorem}[section]

\newtheorem{lemma}[theorem]{Lemma}
\newtheorem{proposition}[theorem]{Proposition}
\newtheorem{definition}[theorem]{Definition}

\theoremstyle{remark}
\newtheorem{remark}{Remark}[section]

\newcommand{\be}[1]{\begin{equation}\label{#1}}
\newcommand{\ee}{\end{equation}}
\numberwithin{equation}{section}

\newcommand{\ba}[1]{\begin{align}\label{#1}}
\newcommand{\ea}{\end{align}}
\numberwithin{equation}{section}

\newcommand{\ben}{\begin{equation*}}
\newcommand{\een}{\end{equation*}}
\numberwithin{equation}{section}

\newcommand{\calP}{\mathcal{P}}

\newcommand{\calS}{\mathcal{S}}

\newcommand{\bbP}{\mathbb{P}}

\newcommand{\bbR}{\mathbb{R}}

\newcommand{\sfC}{{\sf C}}

\newcommand{\ep}{\varepsilon}

\newcommand{\n}{\mathbf n}

\newcommand{\rk}[1]{\bgroup\color{red}%
  \par\medskip\hrule\smallskip%
  \noindent\textbf{#1}%
  \par\smallskip\hrule\medskip\egroup}

\usepackage{hyperref}
\usepackage{mathtools}
\mathtoolsset{showonlyrefs}

\renewcommand{\P}{\bbP_\beta}
\newcommand{\lr}[1][]{\stackrel{#1}\longleftrightarrow}

\newcommand{\nlr}[1][]{\overset{#1}{\not\longleftrightarrow}}

\title{A new proof of the sharpness of the phase transition for Bernoulli percolation and the Ising model}
\author{Hugo Duminil-Copin and Vincent Tassion}
\date{\today}

\begin{document}
\maketitle

\begin{abstract}
  We provide a new proof of the sharpness of the phase transition for Bernoulli
  percolation and the Ising model. The proof applies to infinite-range models on
  arbitrary locally finite transitive infinite graphs.
  
  For Bernoulli percolation, we prove finiteness of the susceptibility in the
  subcritical regime $\beta<\beta_c$, and the mean-field lower bound
  $\bbP_\beta[0\longleftrightarrow\infty]\ge (\beta-\beta_c)/\beta$ for
  $\beta>\beta_c$. For finite-range models, we also prove that for any
  $\beta<\beta_c$, the probability of an open path from the origin to distance
  $n$ decays exponentially fast in $n$.

  For the Ising model, we prove finiteness of the susceptibility for
  $\beta<\beta_c$, and the mean-field lower bound $\langle
  \sigma_0\rangle_\beta^+\ge \sqrt{(\beta^2-\beta_c^2)/\beta^2}$ for
  $\beta>\beta_c$. For finite-range models, we also prove that the two-point
  correlation functions decay exponentially fast in the distance for
  $\beta<\beta_c$.
\end{abstract}

The paper is organized in two sections, one devoted to Bernoulli percolation,
and one to the Ising model. While both proofs are completely independent, we
wish to emphasize the strong analogy between the two strategies.

\paragraph{General notation.} Let $G=(V,E)$ be a locally finite
(vertex-)transitive infinite graph, together with a fixed origin $0\in V$. For
$n\ge0$, let
$$\Lambda_n:=\{x\in V:\mathrm{d}(x,0)\le n\},$$ where $\mathrm{d}(\cdot,\cdot)$ is the
graph distance. Consider a set of coupling constants $(J_{x,y})_{x,y\in V}$ with
$J_{x,y}=J_{y,x}\ge0$ for every $x$ and $y$ in $V$. We assume that the coupling
constants are {\em invariant} with respect to some transitively acting group.
More precisely, there exists a group $\Gamma$ of automorphisms acting
transitively on $V$ such that $J_{\gamma(x),\gamma(y)}=J_{x,y}$ for all $\gamma
\in \Gamma$. We say that $(J_{x,y})_{x,y\in V}$ is {\em finite-range} if there
exists $R>0$ such that $J_{x,y}=0$ whenever $\mathrm{d}(x,y)>R$.

\section{Bernoulli percolation}

\subsection{The main result}
Let $\bbP_\beta$ be the bond percolation measure on $G$ defined as follows: for
$x,y\in V$, $\{x,y\}$ is {\em open} with probability $1-e^{-\beta J_{x,y}}$, and
{\em closed} with probability $e^{-\beta J_{x,y}}$. We say that $x$ and $y$ are
{\em connected in} $S\subset V$ if there exists a sequence of vertices
$(v_k)_{0\le k\le K}$ in $S$ such that $v_0=x$, $v_K=y$, and $\{v_k,v_{k+1}\}$
is open for every $0\le k<K$. We denote this event by
$x\stackrel{S}{\longleftrightarrow} y$. For $A \subset V$, we write
$x\stackrel{S}{\longleftrightarrow} A$ for the event that $x$ is connected in
$S$ to a vertex in $A$. If $S=V$, we drop it from the notation. Finally, we set
$0\longleftrightarrow \infty$ if $0$ is connected to~$\Lambda_n^c$ for all
$n\ge1$. The critical parameter is defined by $$\beta_c:=\inf\{\beta\ge0:
\bbP_\beta[0\longleftrightarrow\infty]>0\}.$$

\begin{theorem}\label{thm:perco}

\begin{enumerate}
\item\label{item:1} For $\beta>\beta_c$, $\bbP_\beta[0\longleftrightarrow\infty]\ge \frac{\beta-\beta_c}{\beta} $.
\item\label{item:2} For $\beta<\beta_c$, the susceptibility is finite, i.e.
$$\sum_{x\in V}\bbP_\beta[0\longleftrightarrow x]<\infty.$$
\item\label{item:3} If $(J_{x,y})_{x,y\in V}$ is finite-range, then for any $\beta<\beta_c$, there exists $c=c(\beta)>0$ such that
$$\bbP_\beta[0\longleftrightarrow \Lambda_n^c]\le e^{-c n}\quad\text{ for all }n\ge0.$$
\end{enumerate}
\end{theorem}

\bigbreak
Let us describe the proof quickly.
For $\beta>0$ and a finite subset $S$ of $V$, define 
\begin{equation}\label{eq:1}
\varphi_\beta(S):=\sum_{x\in S}\sum_{y\notin S}(1-e^{-\beta J_{x,y}})\bbP_\beta\big(0\stackrel{S}{\longleftrightarrow }x\big).
\end{equation}
This quantity can be interpreted as the expected number of open edges
on the ``external boundary'' of $S$ that are connected to $0$ by an
open path of vertices in $S$. Also introduce
\begin{equation}
  \label{eq:2}
    \tilde\beta_c:=\sup\{\beta\ge0: \varphi_\beta(S) <1\text{ for some
      finite $S\subset V$ containing $0$} \}.
\end{equation}
In order to prove Theorem~\ref{thm:perco}, we show that
Items~\ref{item:1}, \ref{item:2} and~\ref{item:3} hold with
$\tilde{\beta_c}$ in place of $\beta_c$. This directly implies that
$\tilde{\beta_c}=\beta_c$, and thus Theorem~\ref{thm:perco}.

 The quantity $\varphi_\beta(S)$ appears
naturally when differentiating the probability $\P[0\lr\Lambda_n^c]$
with respect to $\beta$. Indeed, a simple computation presented in
Lemma~\ref{lem:meanField} provides the following differential inequality \begin{equation}
   \label{eq:3}
   \frac d{d\beta} \P[0\lr\Lambda_n^c]\ge \tfrac1\beta \inf_{\substack{S\subset\Lambda_n\\0\in S}}\varphi_\beta(S) \cdot (1-\P[0\lr\Lambda_n^c]).
  \end{equation}
By integrating \eqref{eq:3} between $\tilde \beta_c$ and $\beta>\tilde\beta_c$ and
then letting $n$ tend to infinity, we obtain $\P[0\lr\infty]\ge
\frac{\beta-\tilde\beta_c}{\beta}$. 

Now consider $\beta<\tilde\beta_c$. The existence of a finite set $S$
containing the origin such that $\varphi_\beta(S)<1$, together with
the BK-inequality, imply that the expected size of the cluster the
origin is finite. 

\subsection{Comments and consequences}\label{sec:comment}

\begin{description}
\item[Bibliographical comments.] Theorem~\ref{thm:perco} was first
  proved in \cite{AizBar87} and \cite{Men86} for Bernoulli percolation
  on the $d$-dimensional hypercubic lattice. The proof was extended to
  general quasi-transitive graphs in \cite{antunovic2008sharpness}.
  The first item was  proved in \cite{ChaCha87}.

\item[Nearest-neighbor percolation.] We recover easily the standard results for
  nearest-neighbor model by setting $J_{x,y}=0$ if $\{x,y\}\notin E$,
  $J_{x,y}=1$ if $\{x,y\}\in E$, and $p=1-e^{-\beta}$. In this context, one can
  obtain the inequality $\mathbb P_p[0\lr\infty]\ge\frac{p-p_c}{p(1-p_c)}$ for
  $p\ge p_c$ by introducing
  $$\varphi_p(S)=p\sum_{x\in S}\sum_{\substack{y\notin S\\ \{x,y\}\in E}}\mathbb P_p[0\lr[S]x].$$
   This lower
  bound is slightly better than Item~\ref{item:1} of
  Theorem~\ref{thm:perco} and is provided by little modifications in
  our proof (see \cite{DumTas15} for a presentation of the proof in
  this context).

\item[Site percolation.] As in \cite{AizBar87}, the proof may be adapted to site percolation on transitive graphs. In this context, one can
  obtain the inequality $\mathbb
  P_p[0\lr\infty]\ge \frac1{d-1}\frac{p-p_c}{1-p_c}$ ($d$ is the degree of $G$) for $p\ge p_c$ by introducing
  $$\varphi_p(S)=\sum_{x\in S}\sum_{\substack{y\notin S\\ \{x,y\}\in E}}\mathbb P_p[0\lr[S]x].$$

\item[Finite susceptibility against exponential decay.] Finite
  susceptibility does not always imply exponential decay of
  correlations for infinite-range models. Conversely, on graphs with
  exponential growth, exponential decay does not imply finite
  susceptibility. Hence, in general, the second condition of
  Theorem~\ref{thm:perco} is neither weaker nor stronger than the
  third one.

\item[Percolation on the square lattice.] On the square lattice, the inequality $p_c\ge1/2$ was first obtained by Harris in \cite{Har60} (see also the short proof of Zhang presented in
  \cite{Gri99a}). The other inequality $p_c\le 1/2$ was first proved by Kesten in \cite{Kes80} using a delicate geometric construction involving crossing events. Since then, many other proofs invoking exponential decay in the subcritical phase (see \cite{Gri99a}) or sharp threshold arguments (see e.g. \cite{BolRio06c}) have been found. Here, Theorem~\ref{thm:perco} provides a short proof of exponential decay and therefore a short alternative to these proofs. For completeness, let us sketch how exponential decay implies that $p_c\le1/2$: item~\ref{item:3}
  implies that for $p<p_c$ the probability of an open path from left to right in a $n$ by $n$
  square tends to $0$ as $n$ goes to infinity. But self-duality implies that this does not happen when
  $p=1/2$, thus implying that $p_c\le 1/2$. 
  
\item[Lower bound on $\beta_c$.] Since
  $\varphi_\beta(\{0\})=\sum_{y\in V}1-e^{-\beta J_{0,y}}$, we obtain
  a lower bound on $\beta_c$ by taking the solution of the
  equation $\sum_{y\in V}1-e^{-\beta J_{0,y}}=1$.
\item[Behaviour at $\beta_c$.] Under the hypothesis that $\sum_{y\in V}
  J_{0,y}<\infty$, the set \[\{\beta\ge0: \varphi_\beta(S) <1\text{ for some
    finite $S\subset V$ containing $0$ }\} \] defining $\tilde \beta_c$ in
  Equation~\eqref{eq:2} is open. In particular, we have that at
  $\beta=\beta_c=\tilde\beta_c$, $\varphi_\beta(S)\ge 1$ for every finite $S\ni
  0$. This implies the following classical result.
 \begin{proposition}[\cite{AizNew84}]\label{prop:a}
 We have $\displaystyle \sum_{x\in V}\bbP_{\beta_c}[0\longleftrightarrow x]=\infty.$
 \end{proposition}
 \begin{proof}
 Simply write
 $$\Big(\sum_{y\in V}1-e^{-\beta_c J_{0,y}}\Big)\cdot\displaystyle\sum_{x\in V}\bbP_{\beta_c}[0\longleftrightarrow x]\ge \sum_{n\ge1}\varphi_\beta(\Lambda_n)=\infty.$$
 \end{proof}
\item[Semi-continuity of $\beta_c$.] Consider the nearest-neighbor
  model. Since $\tilde{\beta_c}$ is defined in terms of finite sets,
  one can see that $\tilde{\beta_c}$ is lower semi-continuous when seen
  as a function of the graph in the following sense. Let $G$ be an
  infinite locally finite transitive graph. Let $(G_n)$ be a sequence
  of infinite locally finite transitive graphs such that the
  balls of radius $n$ around the origin in $G_n$ and $G$ are the same.
  Then,
  \begin{equation}
    \liminf \tilde\beta_c(G_n)\ge \tilde\beta_c(G).\label{eq:4}
  \end{equation}
  The equality $\beta_c=\tilde\beta_c$ implies that the semi-continuity
  \eqref{eq:4} also holds for $\beta_c$ (this also followed from \cite{Ham57} and the exponential decay in subcritical, but the definition of $\tilde\beta_c$ illustrates this property readily). The locality conjecture, due to
  Schramm and presented in \cite{BenNacPer11}, states that for any $\ep>0$, the map $G\mapsto\beta_c(G)$ should be continuous on the set of graphs with $\beta_c<1-\ep$. The discussion above shows that the hard
  part in the locality conjecture is the upper semi-continuity. 
\item[Dependent models.] For dependent percolation models, the proof
  does not extend in a trivial way, mostly due to the fact that the BK
  inequality is not available in general. Nevertheless, this new
  strategy may be of some use. For instance, for random-cluster models
  on the square lattice, a proof (see \cite{DumSidTas13}) based on the
  strategy of this paper and the parafermionic observable offers an
  alternative to the standard proof of \cite{BefDum12} based on sharp
  threshold theorems.
\item[Oriented percolation.] The proof applies mutatis mutandis to oriented percolation. 
\item[Percolation with a magnetic field.] In \cite{AizBar87}, the authors
  consider a percolation model with magnetic field defined as follows. 
  Add a {\em ghost vertex} $g\notin V$ and
  consider that $\{x,g\}$ is open with probability $1-e^{-h}$, independently for
  any $x\in V$. Let $\bbP_{\beta,h}$ be the measure obtained from $\bbP_\beta$
  by adding the edges $\{x,g\}$. An important results in \cite{AizBar87} is the following mean-field lower bound which is instrumental in the study of percolation in high dimensions (see e.g. \cite{AizNew84}).
  
  \begin{proposition}[\cite{AizBar87}]\label{prop:magnetic field}
  There exists a constant $c>0$ such that for any $h>0$, 
  $$\bbP_{\beta_c,h}[0\longleftrightarrow g]\ge c\sqrt h.$$
  \end{proposition}
  In Section~\ref{sec:90}, we provide a short proof of this proposition, using the
  same strategy as in our proof of Theorem~\ref{thm:perco}.
  
\end{description}

\subsection{Proof of  Item \ref{item:1}}

In this section, we prove that for every $\beta\ge \tilde \beta_c$,
\begin{equation}
  \label{eq:5}
  \bbP_\beta[0\longleftrightarrow\infty]\ge
  \frac{\beta-\tilde{\beta_c}}{\beta}.
\end{equation}
Let us start by the following lemma.
 \begin{lemma}\label{lem:meanField}
Let $\beta>0$ and $\Lambda\subset V$ finite,
  \begin{equation}
    \label{eq:6}
    \frac d{d\beta} \P[0\lr\Lambda^c]\ge \tfrac1\beta
    \inf_{\substack{S\subset\Lambda\\0\in S}}\varphi_\beta(S) \cdot\big(1 - \P[0\lr\Lambda^c]\big).
  \end{equation}
  \end{lemma}
  Before proving this lemma, let us see how it implies~\eqref{eq:5}.
  By setting $f(\beta)=\P[0\lr\Lambda^c]$ in~\eqref{eq:6}, and
  observing that $\varphi_\beta(S)\ge1$ for any $\beta>\tilde\beta_c$, we obtain
  the following differential inequality:
\begin{equation}\label{eq:7}\frac{f'(\beta)}{1-f(\beta)}\ge\frac1\beta,\quad\text{for
    $\beta\in (\tilde\beta_c,\infty).$}\end{equation}
Integrating \eqref{eq:7} between $\tilde\beta_c$ and $\beta$ implies that
$\P[0\lr\Lambda^c]=f(\beta)\ge \frac{\beta-\tilde\beta_c}{\beta}$ for every
$\Lambda\subset V$. By letting $\Lambda$ tend to $V$, we obtain \eqref{eq:5}.
\begin{proof}[Proof of Lemma~\ref{lem:meanField}]
  Let $\beta> 0$ and $\Lambda$. Define the following
  random subset of $\Lambda$:
$$\calS:=\{x\in\Lambda\text{ such that } x\not\longleftrightarrow \Lambda^c\}.$$ 
Recall that $\{x,y\}$ is pivotal for the configuration $\omega$ and
the event $\{0\longleftrightarrow\Lambda^c\}$ if
$\omega_{\{x,y\}}\notin \{0\longleftrightarrow\Lambda^c\}$ and
$\omega^{\{x,y\}}\in \{0\longleftrightarrow\Lambda^c\}$. (The
configuration $\omega_{\{x,y\}}$, resp.\@ $\omega^{\{x,y\}}$,
coincides with $\omega$ except that the edge $\{x,y\}$ is closed,
resp.\@ open.)

Russo's formula (\cite{Rus78} or \cite[Section 2.4]{Gri99a}) implies that
\begin{align}\frac{d}{d\beta}\bbP_\beta[0\longleftrightarrow\Lambda^c] & =\sum_{\{x,y\}}J_{x,y}\bbP_\beta[\{x,y\}\text{ pivotal}]\label{eq:bateau}\\
  &\ge\tfrac1{\beta} \sum_{\{x,y\}}(1-e^{-\beta
    J_{x,y}})\bbP_\beta[\{x,y\}\text{
    pivotal\@ and }0\nlr\Lambda^c]\\
  &\ge\tfrac1{\beta} \sum_{S\ni 0}\sum_{\{x,y\}}(1-e^{-\beta
    J_{x,y}})\bbP_\beta[\{x,y\}\text{ pivotal\@ and }\calS=S].
\end{align}
In the second line, we used the inequality $t\ge1-e^{-t}$ for $t\ge0$. Observe
that the event that $\{x,y\}$ is pivotal and $\calS=S$ is nonempty only if $x\in
S$ and $y\notin S$, or $y\in S$ and $x\notin S$. Furthermore, the vertex in $S$
must be connected to 0 in $S$. We can assume without loss of generality that
$x\in S$ and $y\notin S$. Rewrite the event that $\{x,y\}$ is pivotal and
$\calS=S$ as $ \{ 0\lr[S] x\} \cap \{\calS=S\}$. Since the event $\{\calS=S\}$
and $\{0\lr[S]x\}$ are measurable with respect to the state of edges having one
endpoint in $V\setminus S$, and edges having both endpoints in $S$ respectively.
Therefore, the two events above are independent. Thus,
\[\bbP_\beta[\{x,y\}\text{ pivotal\@ and }\calS=S]=\P[ 0\lr[S] x] \P[\calS=S]. \]
Plugging this equality in the computation above, we obtain
\begin{align}\label{eq:8}
  \frac{d}{d\beta}\bbP_\beta[0\longleftrightarrow\Lambda^c]&
  \ge\tfrac1\beta \sum_{S\ni 0}\varphi_\beta(S)\bbP_\beta[\calS=S]\\
  &\ge\tfrac1\beta \Big(\inf_{S\ni 0}\varphi_\beta(S)\Big)\cdot \sum_{S\ni 0}\bbP_\beta[\calS=S].\nonumber
\end{align}
The proof follows readily since 
$$\sum_{S\ni 0}\bbP_\beta[\calS=S]=\bbP_\beta[0\in \calS]=\bbP_\beta[0\not\longleftrightarrow\Lambda^c]=1-\bbP_\beta[0\longleftrightarrow\Lambda^c].$$

\end{proof}
\begin{remark}
  In the proof above, Russo's formula is possibly used in infinite volume, since
  the model can be infinite-range. There is no difficulty resolving this
  technical issue (which does not occur for finite-range) by finite volume
  approximation. The same remark applies below when we use the BK inequality.
\end{remark}

\subsection{Proof of Items~\ref{item:2} and \ref{item:3}}
\label{sec:proof-item-2-3}
In this section, we show that Items~\ref{item:2} and \ref{item:3} in
Theorem~\ref{thm:perco} hold with $\tilde{\beta_c}$ in place of
$\beta_c$.

\begin{lemma}
  \label{lem:finiteCriterion}
  Let $\beta>0$, and $u\in S\subset A$ and $B\cap S=\emptyset$. We have 
\begin{equation}
   \bbP_\beta[u\lr[A] B]\le\sum_{x\in S}\sum_{y\notin
    S}(1-e^{-\beta J_{x,y}})\bbP_\beta[u\lr[S] x]\bbP_\beta[y\lr[A] B].
\end{equation}
\end{lemma}
\begin{proof}[Proof of Lemma~\ref{lem:finiteCriterion}]
Let $u\in S$ and assume that
the event $u\lr[A] B$ holds. Consider an open path
$(v_k)_{0\le k\le K}$ from $u$ to $B$. Since $B\cap S=\emptyset$, one can define the first $k$ such that $v_{k+1}\notin S$. We obtain that the following events occur disjointly (see \cite[Section 2.3]{Gri99a} for a definition of disjoint occurrence):
\begin{itemize}[noitemsep,nolistsep]\item $u$ is connected to $v_k$ in $S$,
\item $\{v_k,v_{k+1}\}$ is open,
\item $v_{k+1}$ is connected to $B$ in $A$.
\end{itemize}
The lemma is then a direct consequence of the BK inequality applied twice ($v_k$ plays the role of $x$, and $v_{k+1}$ of $y$).
\end{proof}

  Let us now prove the second item of Theorem~\ref{thm:perco}. Fix $\beta<\tilde\beta_c$ and $S$ such that
  $\varphi_\beta(S)<1$. For $\Lambda\subset V$ finite, introduce
\[
\chi(\Lambda,\beta):=\max\Big\{\sum_{v\in\Lambda}\bbP_\beta[u\stackrel{\Lambda}\longleftrightarrow
v]~;~u\in\Lambda\Big\}.
\]
For every $u$, let $S_u$ be the image of $S$ by a fixed automorphism
sending 0 to $u$. Lemma~\ref{lem:finiteCriterion} implies that for
every $v\in \Lambda \setminus S_u$,
\begin{equation}
   \bbP_\beta[u\lr[\Lambda] v]\le\sum_{x\in S_u}\sum_{y\notin
    S_u}\bbP_\beta[u\lr[S_u] x](1-e^{-\beta J_{x,y}})\bbP_\beta[y\lr[\Lambda] v].
\end{equation}
Summing over all $v\in \Lambda \setminus S_u$, we find
\[\sum_{v\in  \Lambda \setminus S_u} \bbP_\beta[u\lr[\Lambda] v]\le\varphi_\beta(S)\chi(\Lambda,\beta).\]
Using the trivial bound $\P[u\lr v]\le 1$ for $v\in \Lambda
\cap S_u$, we obtain
\begin{equation*}
    \sum_{v\in\Lambda}\bbP_\beta[u\stackrel{\Lambda}\longleftrightarrow
    v]\le |S|+\varphi_\beta(S)\chi(\Lambda,\beta).
\end{equation*}
Optimizing over $u$, we deduce that 
\begin{equation*}
  \label{eq:9}
  \chi(\Lambda,\beta)\le \frac{|S|}{1-\varphi_\beta(S)}
\end{equation*}
which implies in particular that\[\sum_{x\in\Lambda}\bbP_\beta[0\stackrel{\Lambda}\longleftrightarrow x]\le\frac{|S|}{1-\varphi_\beta(S)}.\]
The result follows by taking the limit as $\Lambda$ tends to $V$.
\bigbreak We now turn to the proof of the third item of
Theorem~\ref{thm:perco}. A similar proof was used in \cite{Ham57}. Let $R$ be
the range of the $(J_{x,y})_{x,y\in V}$, and let $L$ be such that $S\subset
\Lambda_{L-R}$. 
Lemma~\ref{lem:finiteCriterion} implies that for $n\ge L$,
\begin{align*}\bbP_\beta[0\longleftrightarrow \Lambda_n^c]&\le \sum_{x\in S}\sum_{y\notin S}(1-e^{-\beta J_{x,y}})\bbP_\beta[0\stackrel{S}{\longleftrightarrow} x]\bbP_\beta[y\longleftrightarrow\Lambda_n^c]\\
&\le\varphi_\beta(S)\bbP_\beta[0\longleftrightarrow\Lambda_{n-L}^c].\label{eq:10}\end{align*}
In the last line, we used that $y$ is connected to distance larger than or equal to $n-L$ since $1-e^{-\beta J_{x,y}}=0$ if $x\in S$ and $y$ is not in $\Lambda_L$. By iterating, this immediately implies that 
$$\bbP_\beta[0\longleftrightarrow \Lambda_n^c]\le \varphi_\beta(S)^{\lfloor n/L\rfloor}.$$

\subsection{Proof of Proposition~\ref{prop:magnetic field}}\label{sec:90}
Let us introduce $M(\beta,h)=\bbP_{\beta,h}[0\longleftrightarrow
  g]$.   \begin{lemma}[\cite{AizBar87}]\begin{equation}\label{eq:hu}
  \frac{\partial M}{\partial\beta}\le \big(\sum_{x\in V}J_{0,x}\big)\, M\frac{\partial M}{\partial
    h}.
\end{equation}\end{lemma}
\begin{proof}
Consider a finite subset $\Lambda$ of $V$. Russo's formula leads to the following version of \eqref{eq:bateau}:
\begin{equation}\label{eq:bateau2}\frac{\partial \bbP_{\beta,h}[0\longleftrightarrow \Lambda^c\cup\{g\}]}{\partial\beta}=\sum_{\{x,y\}}J_{x,y}\,\bbP_{\beta,h}[\{x,y\}\text{ pivotal}].\end{equation}
The edge $\{x,y\}$ is pivotal if, without using $\{x,y\}$,  one of the two vertices is connected to $0$
but not to $\Lambda^c\cup\{g\}$, and the other one to
$\Lambda^c\cup\{g\}$. Without loss of generality, let us assume that $x$ is
connected to 0, and $y$ is not. Conditioning on the set
$$\calS=\big\{z\in\Lambda:z\longleftrightarrow \Lambda^c\cup\{g\}\text{ without using }\{x,y\}\big\},$$ we obtain
$$\bbP_{\beta,h}[\{x,y\}\text{ pivotal}]\le \bbP_{\beta,h}[y\leftrightarrow \Lambda^c\cup\{g\}] \cdot\bbP_{\beta,h}[0\longleftrightarrow x,0\not\longleftrightarrow \Lambda^c\cup\{g\}].$$
Plugging this inequality in \eqref{eq:bateau} and letting $\Lambda$ tend to $V$, we find
\begin{equation}\label{eq:bateau2}\frac{\partial M}{\partial\beta}\le \Big(\sum_{\{0,y\}}J_{0,y}\Big)\, M\,\Big(\sum_{x\in V}\bbP_{\beta,h}[0\longleftrightarrow x,0\not\longleftrightarrow g]\Big).\end{equation}
We conclude by observing that if $\sfC$ denotes the cluster of 0 in $V$, we find
\begin{align*}
\frac{\partial M}{\partial
    h}&=\frac{\partial }{\partial
    h}\Big(1-\sum_{n=0}^\infty \bbP_{\beta,h}[|\sfC|=n]e^{-nh}\Big)=\sum_{n=0}^\infty n\bbP_{\beta,h}[|\sfC|=n]e^{-nh}\\
    &=\sum_{n=0}^\infty \sum_{x\in V}\bbP_{\beta,h}[0\longleftrightarrow x,|\sfC|=n]e^{-nh}=\sum_{x\in V}\bbP_{\beta,h}[0\longleftrightarrow x,0\not\longleftrightarrow g].\end{align*}
\end{proof}
Another
differential inequality, which is harder to obtain, usually complements \eqref{eq:hu}:
\begin{equation}\label{eq:11}M\le h\frac{\partial M}{\partial h}+M^2+\beta M\frac{\partial M}{\partial \beta}.\end{equation}
This other inequality may be avoided using the following observation. The
differential inequality \eqref{eq:6} is satisfied with $\P[0\lr\Lambda^c]$
replaced by $\bbP_{\beta,h}[0\lr\{g\}\cup\Lambda^c]$, thus giving us for $\beta\ge\beta_c$ and $h\ge0$,
\begin{align*}\frac{\partial M}{\partial\beta}\ge\tfrac1{\beta}\,(1-M)\end{align*}
(at $\beta=\beta_c$, we use the fact that $\varphi_{\beta_c}(S)\ge1$ for every finite $S\ni 0$, see the comment before Proposition~\ref{prop:a}). When $\beta\ge\beta_c$ this implies that
\begin{equation}\label{eq:12}1-M~\le~ \beta\,\frac{\partial M}{\partial\beta}~\le~ \beta\,\big(\sum_{x\in V}J_{0,x}\big) \, M\frac{\partial M}{\partial h},\end{equation}
which immediately implies the following mean-field lower bound: there exists a
constant $c>0$ such that for any $h>0$,
$$\bbP_{\beta_c,h}[0\longleftrightarrow g]=M(\beta_c,h)~\ge~ c\sqrt h.$$

\begin{remark}
While \eqref{eq:12} is slightly shorter to obtain that \eqref{eq:11}, the later is very useful when trying to obtain an upper bound on $M(\beta_c,h)$. 
\end{remark}
\section{The Ising model}
\subsection{The main result}
For a finite subset $\Lambda$ of $V$, consider a spin configuration $\sigma=(\sigma_x:x\in \Lambda)\in\{-1,1\}^\Lambda$. For $\beta>0$ and $h\in\bbR$, introduce the Hamiltonian
\begin{align*}  
H_{\Lambda,\beta,h}(\sigma)~&:=~-~\beta \sum_{x,y\in\Lambda}J_{x,y}\sigma_x\sigma_y~-~ h \sum_{x\in \Lambda}  \sigma_x. 
\end{align*}
Define the Gibbs measures on $\Lambda$ with free boundary conditions, inverse-temperature $\beta$ and external field $h\in\bbR$ by the formula
$$\langle f\rangle_{\Lambda,\beta,h}=\frac{\displaystyle\sum_{\sigma\in\{-1,1\}^{\Lambda}}f(\sigma)\exp[-H_{\Lambda,\beta,h}(\sigma)]}{\displaystyle\sum_{\sigma\in\{-1,1\}^{\Lambda}}\exp[-\beta H_{\Lambda,\beta,h}(\sigma)]}$$
for $f:\{-1,1\}^{\Lambda}\longrightarrow \bbR$. 
Let the infinite-volume Gibbs measure $\langle \cdot\rangle_{\beta,h}$
be the weak limit of $\langle \cdot\rangle_{\Lambda,\beta,h}$ as
$\Lambda\nearrow V$. Also write $\langle \cdot\rangle_{\beta}^+$ for the weak limit of $\langle \cdot\rangle_{\beta,h}$ as $h\searrow 0$.

Introduce
$$\beta_c:=\inf\{\beta>0:\langle\sigma_0\rangle_\beta^+>0\}.$$
\begin{theorem}\label{thm:Ising}
\begin{enumerate}
\item\label{item:4} For $\beta>\beta_c$, $\langle\sigma_0\rangle_\beta^+\ge\sqrt{\frac{\beta^2-\beta_c^2}{\beta^2}}$.
\item\label{item:5} For $\beta<\beta_c$, the susceptibility is finite, i.e.
$$\sum_{x\in V}\langle\sigma_0\sigma_x\rangle_\beta^+<\infty.$$
\item\label{item:6} If $(J_{x,y})_{x,y\in V}$ is finite-range, then for any $\beta<\beta_c$, there exists $c=c(\beta)>0$ such that
$$\langle\sigma_0\sigma_x\rangle_\beta^+\le e^{-c \mathrm{d}(0,x)}\quad\text{ for all }x\in V.$$
\end{enumerate}
\end{theorem}
This theorem was first proved in \cite{AizBarFer87} for the Ising model on
the $d$-dimensional hypercubic lattice. The proof presented here
improves the constant in the mean-field lower bound, and extends to
general transitive graphs.\bigbreak
The proof of Theorem~\ref{thm:Ising} follows
closely the proof for percolation.  
For $\beta>0$ and a finite subset $S$ of $V$, define
\begin{equation}
\varphi_S(\beta):=\sum_{x\in S}\sum_{y\in V\setminus S}\tanh(\beta J_{x,y})\langle \sigma_0\sigma_x\rangle_{S,\beta,0},
\end{equation}
which bears a resemblance to \eqref{eq:1}. Similarly to \eqref{eq:2}, set 
$$\tilde\beta_c:=\sup\{\beta\ge0:\varphi_\beta(S)<1\text{ for some
  finite }S\subset V\text{ containing 0}\}.$$

In order to prove Theorem~\ref{thm:Ising}, we show that
Items~\ref{item:4}, \ref{item:5} and~\ref{item:6} hold with
$\tilde{\beta_c}$ in place of $\beta_c$. This directly implies that
$\tilde{\beta_c}=\beta_c$, and thus Theorem~\ref{thm:Ising}. The proof
of Theorem~\ref{thm:Ising} proceeds in two steps.

As for percolation, the quantity $\varphi_\beta(S)$ appears naturally
in the derivative of a ``finite-volume  approximation'' of
$\langle\sigma_0\rangle_{\beta,h}$. Roughly speaking (see Lemma~\ref{lem:ising:meanField} for a
precise statement), one obtains a finite-volume version of the
following inequality:
  $$\frac d{d\beta}\langle\sigma_0\rangle_{\beta,h}^2\ge \tfrac2{\beta}\, \inf_{S\ni 0}\varphi_\beta(S)\cdot \big(1-\langle\sigma_0\rangle_{\beta,h}^2\big).$$
  This inequality implies, for every $\beta>\tilde{\beta_c}$,
  \begin{equation}
    \label{eq:13}
    \langle\sigma_0\rangle_{\beta,h}
    \ge\textstyle\sqrt{\frac{\beta^2-\tilde{\beta_c^2}}{\beta^2}}
  \end{equation}
  and therefore Item \ref{item:4} by letting $h$ tend to 0.
  
  The remaining items follow from an improved Simon's inequality,
  proved below.
 
 \begin{remark}The proof uses the random-current representation. In this context, the derivative of $\langle\sigma_0\rangle_{\beta,h}^2$ has an interpretation which is very close to the differential inequality \eqref{eq:6}. In some sense, percolation is replaced by the trace of the sum of two independent random sourceless currents. Furthermore, the strong Simon's inequality plays the role of the BK inequality
for percolation. \end{remark}

\subsection{Comments and consequences}

\begin{enumerate}
\item The random-cluster model (also called Fortuin-Kasteleyn percolation) with cluster weight $q=2$ is naturally coupled to the Ising model (see \cite{Gri06} for details). The previous theorem implies exponential decay in the subcritical phase for this model. 

\item Exactly like in the case of Bernoulli percolation, the critical parameter of the random-cluster model on the square lattice with $q=2$ can be proved to be equal to $\sqrt 2/(1+\sqrt 2)$ using the exponential decay in the subcritical phase together with the self-duality.

\item The previous item together with the coupling with the Ising model implies that $\beta_c=\frac12\log(1+\sqrt2)$ on the square lattice (see \cite{Ons44,BefDum12a} for alternative proofs).

\item Exactly as for Bernoulli percolation, we get that $\varphi_{\beta_c}(S)\ge1$ for any finite set $S\ni0$, which implies the following classical proposition. 
\begin{proposition}[\cite{Sim80}]\label{prop:b}
We have
$\displaystyle \sum_{x\in V}\langle\sigma_0\sigma_x\rangle_{\beta_c}^+=\infty.$
\end{proposition}
\begin{proof}
Use Griffiths' inequality \eqref{eq:14} below to show that for $x\in \Lambda_n\setminus \Lambda_{n-1}$,
$$\langle\sigma_0\sigma_x\rangle_{\beta_c}^+\ge \langle\sigma_0\sigma_x\rangle_{\beta_c,0}\ge \langle\sigma_0\sigma_x\rangle_{\Lambda_n,\beta_c,0}$$ so that
$$\Big(\sum_{y\in V}\tanh(\beta J_{0,y})\Big)\cdot\sum_{x\in V}\langle\sigma_0\sigma_x\rangle_{\beta_c}^+\ge \sum_{n\ge1}\varphi_{\beta_c}(\Lambda_n)=\infty.$$\end{proof}
\item The equality $\beta_c=\tilde\beta_c$ implies that $\beta_c$ is lower semi-continuous with respect to the graph (see the discussion for Bernoulli percolation).

\item In \cite{AizBarFer87}, the authors also prove the following result.
  
  \begin{proposition}[\cite{AizBarFer87}]\label{prop:field2}
  There exists a constant $c>0$ such that for any $h>0$, 
  $$\langle\sigma_0\rangle_{\beta_c,h}\ge ch^{1/3}.$$
  \end{proposition}
 We present in Section~\ref{sec:901} a short proof of this proposition, using
 the same strategy as in our proof of Proposition~\ref{prop:magnetic field}.

\end{enumerate}

\subsection{Preliminaries}
\label{sec:rand-curr-repr}
\paragraph{Griffiths' inequality.}
The following is a standard consequence of the second Griffiths' inequality
\cite{Gri67}: for $\beta>0$, $h\ge0$ and $S\subset\Lambda$ two finite subsets of
$V$, 
\begin{equation}
  \label{eq:14}
  \langle \sigma_0 \rangle_{S,\beta,h} \leq \langle \sigma_0\rangle_{\Lambda,\beta,h}.
\end{equation}

\paragraph{Random-current representation.}
This section presents a few basic facts on the random-current representation. We
refer to \cite{Aiz82,AizFer86,AizBarFer87} for details on this representation.

Let $\Lambda$ be a finite subset of $V$ and $S\subset\Lambda$. We consider an
additional vertex $g$ not in $V$, called the \emph{ghost vertex}, and
write $\calP_2(S\cup \{g\})$ for the set of pairs
$\{x,y\}$, $x,y\in S\cup \{g\}$. We also define $J_{x,g}=h/\beta$ for every $x\in
\Lambda$.
\begin{definition}
  A {\em current} $\n$ on $S$ (also called a current configuration) is a
  function from $\calP_2(S\cup \{g\})$ to $\{0,1,2,...\}$. A {\em source} of
  $\n=(\n_{x,y}:\{x,y\}\in \calP_2(S\cup \{g\}))$ is a vertex $x\in  S\cup\{g\}$ for which
  $\sum_{y\in S}{\n}_{x,y}$ is odd. The set of sources of $\n$ is denoted by
  $\partial\n$. We say that $x$ and $y$ are connected in $\n$ (denoted by
  $x\lr[\n]y$) if there exists a sequence of vertices $v_0,v_1,\ldots, v_K$ in
  $S\cup \{g\}$ such that $v_0=x$, $v_K=y$ and $\n_{v_k,v_{k+1}}>0$ for every
  $0\le k<K$.
\end{definition}
For a finite subset $\Lambda$ of $V$ and a current $\n$ on $\Lambda$, define
$$w(\n)=w(\n,\beta,h):=\prod_{\{x,y\}\in\calP_2(\Lambda\cup\{g\})}\frac{(\beta
  J_{x,y})^{\n_{x,y}}}{\n_{x,y}!}.$$ 
  From now on, we will write $\sum_{\partial \n=A}$ for the sum running on currents on $S$ with sources $A$. Sometimes, the current $\n$ will be on $S'\subset S$ (and therefore the sum will run on such currents), but this will be clear from context.
  \bigbreak
  An important property of random currents is the following:
for  every subset $A$ of $\Lambda$, we have
\begin{align}
  \label{eq:15}
  \big\langle \prod_{a\in
    A}\sigma_a\big\rangle_{\Lambda,\beta,h}&=\begin{cases}
   \displaystyle   \frac{\sum_{\partial \n=A} w(\n)}{\sum_{\partial\n=\emptyset}w(\n)}&\text{if
      $A$ is even},\\
      \\
   \displaystyle  \frac{\sum_{\partial \n=A\cup\{g\}} w(\n)}{\sum_{\partial\n=\emptyset}w(\n)}&\text{if
      $A$ is odd}.\end{cases}
\end{align}

We will use the following standard lemma on random currents.
\begin{lemma}[Switching Lemma, {\cite[Lemma 3.2]{Aiz82}}]
  Let $A \subset \Lambda$ and $u,v\in \Lambda\cup\{g\}$. Let $F$ be a function from the set of currents on $\Lambda$ to $\bbR$. We have
  \begin{equation}
    \label{eq:16}
    \sum_{\substack{\partial \n_1= A \Delta \{u,v\}\\ \partial \n_2=
        \{u,v\}}}F(\n_1+\n_2)w(\n_1) w(\n_2)=\sum_{\substack{\partial \n_1=
        A\\ \partial \n_2=\emptyset}}F(\n_1+\n_2)w(\n_1) w(\n_2) \mathbf{I}[u\lr[\n_1+\n_2] v],
  \end{equation}
  where $\Delta$ is the symmetric difference between sets.
\end{lemma}

\paragraph{Backbone representation for random currents.} Fix a finite subset $\Lambda$ of $V$. Choose an arbitrary order of the oriented edges of the lattice. Consider a current $\n$ on $\Lambda$ with $\partial\n=\{x,y\}$. Let $\omega(\n)$ be the edge self-avoiding path from $x$ to $y$ passing only through edges $e$ with $\n_e$ odd which is minimal for the lexicographical order on paths induced by the previous ordering on oriented edges. Such an object is called the {\em backbone} of the current configuration. For the backbone $\omega$ with endpoints $\partial\omega=\{x,y\}$, set
$$\rho_\Lambda(\omega)=\rho_\Lambda(\beta,h,\omega):=\frac{\sum_{\partial\n=\{x,y\}}w(\n)\mathbf{I}[\omega(\n)=\omega]}{\sum_{\partial\n=\emptyset}w(\n)}.$$
The backbone representation has the following properties (see (4.2), (4.7) and (4.11) of \cite{AizFer86} for {\bf P1}, {\bf P2} and {\bf P3} respectively):
\begin{itemize}[noitemsep,nolistsep]
\item[{\bf P1}] $\langle \sigma_x\sigma_y\rangle_{\Lambda,\beta,h}=\sum_{\partial\omega=\{x,y\}}\rho_\Lambda(\omega)$.
\item[{\bf P2}] If the backbone $\omega$ is the concatenation of two backbones $\omega_1$ and $\omega_2$ (this is denoted by $\omega=\omega_1\circ\omega_2$), then
$$\rho_\Lambda(\omega)= \rho_\Lambda(\omega_1)\rho_{\Lambda\setminus\overline\omega_1}(\omega_2),$$
where $\overline\omega_1$ is the set of bonds whose state is determined by the fact that $\omega_1$ is an admissible backbone (this includes bonds of $\omega_1$ together with some neighboring bonds).
\item[{\bf P3}] For the backbone $\omega$ not using any edge outside $T\subset
  \Lambda$, we have $$\rho_\Lambda(\omega)\le \rho_T(\omega).$$
\end{itemize}

\subsection{Proof of Item~\ref{item:4}}
In this section, we prove that for every $\beta\ge \tilde \beta_c$,
\begin{equation}\label{eq:17}
 \langle\sigma_0\rangle_\beta^+
    \ge\textstyle\sqrt{\frac{\beta^2-\tilde{\beta_c^2}}{\beta^2}}.
\end{equation}
In order to do so, we will based our analysis on the following lemma.
\begin{lemma}\label{lem:ising:meanField}
Let $\beta>0$, $h>0$ and $\Lambda$ a finite subset of $V$. Then,
 \begin{equation}
   \label{eq:18}
   \frac d{d\beta} \langle\sigma_0\rangle_{\Lambda,\beta,h}^2 \ge 2c(\Lambda,\beta,h)\,\Big[\tfrac1\beta\inf_{S\ni 0}\varphi_\beta(S)(1-\langle\sigma_0\rangle_{\Lambda,\beta,h}^2)-\epsilon(\Lambda,\beta,h)\Big],\end{equation}
   where $$ c(\Lambda,\beta,h):=\inf_{y\in\Lambda}\frac{\langle\sigma_0\rangle_{\Lambda,\beta,h}}{\langle\sigma_y\rangle_{\Lambda,\beta,h}}$$ and
   \begin{align*}
   \epsilon(\Lambda,\beta,h)&:=\sum_{x\in \Lambda}\sum_{y\in V\setminus \Lambda}J_{x,y}(\langle\sigma_0\sigma_x\rangle_{\Lambda,\beta,h}-\langle\sigma_0\rangle_{\Lambda,\beta,h}\langle\sigma_x\rangle_{\Lambda,\beta,h}).
   \end{align*}
  \end{lemma}
To conclude the proof, fix $\beta_1,\beta_2>\tilde\beta_c$.  Integrating \eqref{eq:18} between $\beta_1$ and $\beta_2$ for $\Lambda$ equal to the box $\Lambda_n$ of size $n$, and then letting $\Lambda_n$ go to infinity, implies that 
  \begin{equation}\label{eq:19}
\langle\sigma_0\rangle_{\beta_2,h}^2-\langle\sigma_0\rangle_{\beta_1,h}^2\ge\int_{\beta_1}^{\beta_2}\tfrac{2}{\beta}(1-\langle\sigma_0\rangle_{\beta,h}^2)d\beta,\end{equation}
where the inequality above follows from Fatou's lemma together with 
\begin{align*}
&\lim_{n\rightarrow\infty}\langle\sigma_0\rangle_{\Lambda_n,\beta_i,h}=\langle\sigma_0\rangle_{\beta_i,h}&\text{(by weak convergence),}\\
&\lim_{n\rightarrow\infty}c(\Lambda_n,\beta,h)=1&\text{(see Remark~\ref{rmk:1} below),}\\
&\lim_{n\rightarrow\infty}\epsilon(\Lambda_n,\beta,h)=0&\text{(see Remark~\ref{rmk:2} below).}\end{align*}
 The proof of \eqref{eq:17} follows easily by letting $h$ tend to 0.
  
\begin{remark}\label{rmk:1}
  To see that $c(\Lambda_n,\beta,h)$ tends to 1, observe that  Griffiths' inequality \eqref{eq:14} implies
  that $\langle\sigma_y\rangle_{\Lambda_n,\beta,h}\le\langle\sigma_0\rangle_{\Lambda_{2n},\beta,h}$
  (we use the invariance under translation and the fact that the translate of $\Lambda_{2n}$ centered at $y$ contains $\Lambda_{n}$). Therefore, for every $n\ge1$,
  we have
  \begin{equation}
    \frac{\langle\sigma_0\rangle_{\Lambda_{n},\beta,h}}{\langle\sigma_0\rangle_{\Lambda_{2n},\beta,h}}
    \le c(\Lambda_n,\beta,h)\le1.\label{eq:23}
  \end{equation}
  Together with the fact that $\langle\sigma_0\rangle_{\Lambda_n,\beta,h}$ tends
  to $\langle\sigma_0\rangle_{\beta,h}$ as $n$ tends to infinity,
   \eqref{eq:23} implies that $c(\Lambda_n,\beta,h)$ tends to 1. 
   \end{remark}
   \begin{remark}\label{rmk:2} To see that $\epsilon(\Lambda_n,\beta,h)$ tends to 0, first observe that the GHS inequality \cite{GriHurShe70} implies that $\langle\sigma_0\rangle_{\Lambda,\beta,h}$ is a concave function of $h$. We deduce that 
   $$\sum_{x\in \Lambda}\langle\sigma_0\sigma_x\rangle_{\Lambda,\beta,h}-\langle\sigma_0\rangle_{\Lambda,\beta,h}\langle\sigma_x\rangle_{\Lambda,\beta,h}=\frac{\partial}{\partial h}\langle\sigma_0\rangle_{\Lambda,\beta,h}\le \frac{\langle\sigma_0\rangle_{\Lambda,\beta,h}}{h}\le \frac1h.
$$
Applied to $\Lambda=\Lambda_n$, this gives in particular that for each $k$,
$$\sum_{x\in \Lambda_{n-k}}\sum_{y\in V\setminus \Lambda_n}J_{x,y}\langle\sigma_0\sigma_x\rangle_{\Lambda_n,\beta,h}-\langle\sigma_0\rangle_{\Lambda_n,\beta,h}\langle\sigma_x\rangle_{\Lambda_n,\beta,h}\le\frac1h\Big(\sum_{y\in V\setminus \Lambda_k}J_{0,y}\Big),$$which can be made arbitrarily small (uniformly in $n$) by setting $k$ large enough. Now, a second use of the GHS inequality \cite{GriHurShe70} implies that 
\begin{align*}\sum_{x\in \Lambda_n\setminus\Lambda_{n-k}}\langle\sigma_0\sigma_x\rangle_{\Lambda_n,\beta,h}-\langle\sigma_0\rangle_{\Lambda_n,\beta,h}\langle\sigma_x\rangle_{\Lambda_n,\beta,h}&\le\frac{\langle\sigma_0\rangle_{\Lambda_{n},\beta,h}-\langle\sigma_0\rangle_{\Lambda_{n},\beta,{\bf h}}}h\\
&\le \frac{\langle\sigma_0\rangle_{\Lambda_{n},\beta,h}-\langle\sigma_0\rangle_{\Lambda_{n-k},\beta,h}}h,\end{align*}
where $\langle\cdot\rangle_{\Lambda_n\beta,{\bf h}}$ is the measure with inverse-temperature $\beta$, and magnetic field $h_x$ depending on $x$ which is equal to $h$ for $x\in\Lambda_{n-k}$ and $0$ in $\Lambda_n\setminus\Lambda_{n-k}$. In the second line, we used Griffiths inequality to show that $\langle\sigma_0\rangle_{\Lambda_{n-k},\beta,h}\le \langle\sigma_0\rangle_{\Lambda_{n},\beta,{\bf h}}$.
 For each fixed $k$, the term on the right converges to 0 as $n$ tends to infinity by weak convergence.
\end{remark}
\bigbreak   
   In order to prove
  Lemma~\ref{lem:ising:meanField}, we use a computation similar to one
  provided in~\cite{AizBarFer87}.

\begin{proof}[Proof of Lemma~\ref{lem:ising:meanField}] Let $\beta> 0$, $h>0$ and a finite subset $\Lambda$ of  $V$. 
  Set 
$$Z:=\sum_{\partial\n=\emptyset}w(\n).$$ 
The
derivative of $\langle \sigma_0\rangle_{\Lambda,\beta,h}$ is given by
the following formula
\begin{equation*}
  \frac{d}{d\beta}\langle
  \sigma_0\rangle_{\Lambda,\beta,h}=\sum_{\{x,y\}\subset\Lambda}J_{x,y}\big(\langle
  \sigma_0\sigma_x\sigma_y \rangle_{\Lambda,\beta,h} -\langle
  \sigma_0 \rangle_{\Lambda,\beta,h} \langle
  \sigma_x\sigma_y \rangle_{\Lambda,\beta,h}\big).
\end{equation*}
 Using \eqref{eq:15} and the
switching lemma, we obtain
\begin{equation*}
   \frac{d}{d\beta}\langle
  \sigma_0\rangle_{\Lambda,\beta,h}=\frac1{Z^2}
  \sum_{\{x,y\}\subset\Lambda}\ J_{x,y}\sum_{\substack{\partial
      \n_1=\{0,g\}\Delta\{x,y\}\\ \partial\n_2=\emptyset}}w(\n_1)w(\n_2)\mathbf
  I[0\nlr[\n_1+\n_2]g].
\end{equation*}
If $\n_1$ and $\n_2$ are two currents such that $\partial
      \n_1=\{0,g\}\Delta\{x,y\}$, $\partial\n_2=\emptyset$ and
      $0$ and $g$ are not connected in $\n_1+\n_2$, then exactly one of these two cases
      holds: $0\lr[\n_1+\n_2]x$ and $y\lr[\n_1+\n_2]g$, or $0\lr[\n_1+\n_2]y$ and
        $x\lr[\n_1+\n_2]g$.
Since the second case is the same as the first one with $x$ and
$y$ permuted, we obtain the following expression, 
\begin{equation}
  \label{eq:20}
   \frac{d}{d\beta}\langle
  \sigma_0\rangle_{\Lambda,\beta,h}=\frac1{Z^2}
  \sum_{\substack{x\in\Lambda\\ y\in\Lambda}}J_{x,y}\delta_{x,y},  
\end{equation}
where 
$$\delta_{x,y}=\displaystyle\sum_{\substack{\partial
      \n_1=\{0,g\}\Delta\{x,y\}\\ \partial\n_2=\emptyset}}w(\n_1)w(\n_2)\mathbf
  I[0\lr[\n_1+\n_2]x, y\lr[\n_1+\n_2]g, 0\nlr[\n_1+\n_2]g]$$
  (see Fig.~\ref{fig:piv} and notice the analogy with the event involved in Russo's formula, namely that the edge $\{x,y\}$ is pivotal, in Bernoulli percolation).
  \bigbreak

  \begin{figure}[htbp]
    \centering
    \includegraphics[width=5cm]{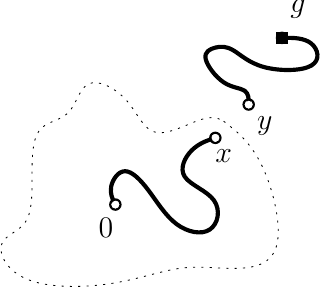}
    \caption{A diagrammatic representation of $\delta_{x,y}$: the solid
      lines represent the backbones, and the dotted line the boundary
      of the cluster of $0$ in $\n_1+\n_2$.}
    \label{fig:piv}
  \end{figure}
Given two currents $\n_1$ and $\n_2$, and $z\in \{0,g\}$, 
define $\calS_z$ to be the set of vertices in $\Lambda
\cup \{g\}$ that are \emph{not} connected to $z$ in $\n_1+\n_2$.
Let us compute $\delta_{x,y}$ by summing over the different possible
values for $\calS_0$:
\begin{align}
  \delta_{x,y}&=\sum_{\substack{S\subset \Lambda\cup \{g\}}}\ \sum_{\substack{\partial
      \n_1=\{0,g\}\Delta\{x,y\}\\ \partial\n_2=\emptyset}}w(\n_1)w(\n_2)\mathbf
  I[\calS_0=S, 0\lr[\n_1+\n_2]x, y\lr[\n_1+\n_2]g, 0\nlr[\n_1+\n_2]g]\\
  &=\sum_{\substack{S\subset \Lambda\cup \{g\}\\ \text{s.t.~}y,g \in
      S\\
      \text{and }0,x\in \Lambda\setminus S}}\ \sum_{\substack{\partial
      \n_1=\{0,g\}\Delta\{x,y\}\\ \partial\n_2=\emptyset}}w(\n_1)w(\n_2)\mathbf
  I[\calS_0=S, y\lr[\n_1+\n_2]g].
\end{align}
Since $0$ and $x$ are not connected to $y$ in $\n_1+\n_2$ (recall that $y\in S$), we deduce that $y$ must be connected to $g$ in $\n_1$ because of the constraints on sources. Thus, the indicator $\mathbf I[y\lr[\n_1+\n_2]g]$ equals 1 for any currents $\n_1$ and $\n_2$ satisfying $\calS_0=S$. Therefore,
\begin{align}
  \delta_{x,y}
  &=\sum_{\substack{S\subset \Lambda\cup \{g\}\\ \text{s.t.~}y,g \in
      S\\
      \text{and }0,x\in \Lambda\setminus S}}\ \sum_{\substack{\partial
      \n_1=\{0,g\}\Delta\{x,y\}\\ \partial\n_2=\emptyset}}w(\n_1)w(\n_2)\mathbf
  I[\calS_0=S].\label{eq:polo}
\end{align}
  Let us now focus on the following claim, which enables us to remove the sources $y$ and $g$.
\bigbreak
\noindent{\em Claim 1: Let $S\subset \Lambda$ containing $y$ and $g$ but neither $x$ nor $0$. We have
\begin{align}\label{eq:claim}\sum_{\substack{\partial
      \n_1=\{0,g\}\Delta\{x,y\}\\ \partial\n_2=\emptyset}}&w(\n_1)w(\n_2)\mathbf
  I[\calS_0=S]
 \\
 &\ge\frac1{\langle \sigma_y \rangle_{\Lambda,\beta,h}}\,\sum_{\substack{\partial
      \n_1=\{0\}\Delta\{x\}\\ \partial\n_2=\emptyset}}w(\n_1)w(\n_2)\mathbf
  I[\calS_0=S, y\lr[\n_1+\n_2]g].\end{align}
}
\bigbreak
\begin{proof}[Proof of Claim 1]
Let 
$$\Theta=\sum_{\substack{\partial
      \n_1=\{0,g\}\Delta\{x,y\}\\ \partial\n_2=\emptyset}}w(\n_1)w(\n_2)\mathbf
  I[\calS_0=S]
.$$
When $\calS_0=S$, the two currents $\n_1$ and $\n_2$ 
vanish on every $\{u,v\}$ with $u\in S$ and $v\notin S$. Thus, for $i=1,2$, we can decompose
$\n_i$ as \[\n_i=\n_i^{S}+\n_i^{\Lambda\setminus S},\] where $\n_i^A$
denotes the current 
$$\n_i^A(\{u,v\})=\begin{cases}\n_i(\{u,v\})&\text{ if }u,v\in A,\\
0&\text{ otherwise}.\end{cases}$$
Note that $\partial\n_i^A=A\cap\partial\n_i$ and $w(\n_i)=w(\n_i^{\Lambda\setminus S})w(\n_i^{S})$.

Since ${\bf I}[\calS_0=S]$ does not depend on $\n_1^{S}$, the decomposition $\n_1=\n_1^{S}+\n_1^{\Lambda\setminus S}$  gives
\begin{align*}\Theta
  &=\sum_{\substack{\partial
      \n_1^{\Lambda\setminus S}=\{0\}\Delta\{x\}\\ \partial\n_2=\emptyset}}w(\n_1^{\Lambda\setminus S})w(\n_2)\mathbf
  I[\calS_0=S]\Big(\sum_{\partial\n_1^{S}=\{y,g\}} w(\n_1^{S})\Big).\end{align*}
Using \eqref{eq:15}, we find \begin{align*}  \Theta&=\sum_{\substack{\partial
      \n_1^{\Lambda\setminus S}=\{0\}\Delta\{x\}\\ \partial\n_2=\emptyset}}w(\n_1^{\Lambda\setminus S})w(\n_2)\mathbf
  I[\calS_0=S]\langle\sigma_y\rangle_{S,\beta,h}\Big(\sum_{\partial\n_1^{S}=\emptyset} w(\n_1^{S})\Big).
 \end{align*}
 Multiply the expression above by $\langle \sigma_y \rangle_{\Lambda,\beta,h}\ge
 \langle\sigma_y\rangle_{S,\beta,h}$ (which follows from \eqref{eq:14}), and then
 decompose $\n_2$ into $\n_2^{S}$ and $\n_2^{\Lambda\setminus S}$ to
 find
\begin{align*}
   \langle \sigma_y \rangle_{\Lambda,\beta,h} \Theta &\ge \sum_{\substack{\partial
      \n_1^{\Lambda\setminus S}=\{0\}\Delta\{x\}\\ \partial\n_2=\emptyset}}w(\n_1^{\Lambda\setminus S})w(\n_2)\mathbf
  I[\calS_0=S]\langle\sigma_y\rangle_{S,\beta,h}^2 \Big(\sum_{\partial\n_1^{S}=\emptyset} w(\n_1^{S})\Big)\\
  &=\sum_{\substack{\partial\n_1^{\Lambda\setminus S}=\{0\}\Delta\{x\}\\
  \partial\n_2^{\Lambda\setminus S}=\emptyset}}w(\n_1^{\Lambda\setminus S})w(\n_2^{\Lambda\setminus S})\mathbf
  I[\calS_0=S]\langle\sigma_y\rangle_{S,\beta,h}^2\\
  &\quad\quad\quad\quad\quad\quad\quad\quad\quad\quad\quad\quad\quad\quad\quad\Big(\sum_{\substack{\partial\n_1^{S}=\emptyset\\
  \partial\n_2^{S}=\emptyset}}w(\n_1^{S})w(\n_2^{S})\Big)\\
    &=\sum_{\substack{\partial\n_1^{\Lambda\setminus S}=\{0\}\Delta\{x\}\\
  \partial\n_2^{\Lambda\setminus S}=\emptyset}}w(\n_1^{\Lambda\setminus S})w(\n_2^{\Lambda\setminus S})\mathbf
  I[\calS_0=S]\\
  &\quad\quad\quad\quad\quad\quad\quad\quad\quad\quad\quad\quad\quad\Big(\sum_{\substack{\partial\n_1^{S}=\{y,g\}\\
  \partial\n_2^{S}=\{y,g\}}}w(\n_1^{S})w(\n_2^{S})\Big)\\
&= \sum_{\substack{\partial
      \n_1=\{0\}\Delta\{x\}\Delta\{y,g\}\\ \partial\n_2=\{y,g\}}}w(\n_1)w(\n_2)\mathbf
  I[\calS_0=S].\end{align*}

  The switching lemma \eqref{eq:16} applied to $F={\bf I}[\calS_0=S]$ implies 
  \begin{align*}
 \langle \sigma_y \rangle_{\Lambda,\beta,h} \,\Theta&\ge\sum_{\substack{\partial
      \n_1=\{0\}\Delta\{x\}\\ \partial\n_2=\emptyset}}w(\n_1)w(\n_2)\mathbf
  I[\calS_0=S, y\lr[\n_1+\n_2]g].\end{align*}
  \end{proof}
Inserting \eqref{eq:claim} into \eqref{eq:polo} gives us
  \begin{align*}
\delta_{x,y}&\ge\frac1{\langle \sigma_y \rangle_{\Lambda,\beta,h}}\,\sum_{\substack{\partial
      \n_1=\{0\}\Delta\{x\}\\ \partial\n_2=\emptyset}}w(\n_1)w(\n_2)\mathbf
  I[y\lr[\n_1+\n_2]g,0\nlr[\n_1+\n_2]g].
\end{align*} 
We now decompose over the possible values of $\calS_g$ (recall that $\calS_g$ is
the set of vertices {\em not} connected to $g$):
\begin{align}
  \delta_{x,y}&\ge\frac1{\langle \sigma_y \rangle_{\Lambda,\beta,h}}\,
  \sum_{S\subset\Lambda}\ \sum_{\substack{\partial
      \n_1=\{0\}\Delta\{x\}\\ \partial\n_2=\emptyset}}w(\n_1)w(\n_2)\mathbf
  I[\calS_g=S,y\lr[\n_1+\n_2]g, 0\nlr[\n_1+\n_2]g]\nonumber\\
  &=\frac1{\langle \sigma_y \rangle_{\Lambda,\beta,h}}\, \sum_{\substack{S\subset\Lambda\\
      \text{s.t.~}0,x\in S\\\text{and }y\in\Lambda\setminus S}}\ \sum_{\substack{\partial
      \n_1=\{0\}\Delta\{x\}\\ \partial\n_2=\emptyset}}w(\n_1)w(\n_2)\mathbf
  I[\calS_g=S].\label{eq:papa}\end{align} In the second line, we used the
constraint on the sources, which implies that $x$ is connected to $0$, and
therefore, belong to $\mathcal S_g$.

We now focus on a second claim, which enables us to remove the sources $0$ and $x$.
 \bigbreak
\noindent{\em Claim 2: Let $S\subset \Lambda$ containing $0$ and $x$ but neither
  $y$ nor $g$. We have
\begin{align*}\label{eq:claim2}\sum_{\substack{\partial
      \n_1=\{0\}\Delta\{x\}\\ \partial\n_2=\emptyset}}&w(\n_1)w(\n_2)\mathbf
  I[\calS_g=S]
 &=\sum_{\substack{\partial
      \n_1=\emptyset\\ \partial\n_2=\emptyset}}w(\n_1)w(\n_2)\langle\sigma_0\sigma_x\rangle_{S,0}\,\mathbf
  I[\calS_g=S].\end{align*}}
  
  \begin{proof}[Proof of Claim 2]
For currents $\n_1$ and $\n_2$ such that $\calS_g=S$, $\n_1$ can be decomposed
as $\n_1=\n_1^S+\n_1^{\Lambda\cup\{g\}\setminus S}$ as we did for $\calS_0=S$ in the previous claim. Using that $w(\n_1)=w(\n_1^S)w(\n_1^{\Lambda\cup\{g\}\setminus S})$ together with the fact that ${\bf I}[\calS_g=S]$ does not depend on $\n_1^S$, we find that
\begin{align*}
\sum_{\substack{\partial
      \n_1=\{0\}\Delta\{x\}\\ \partial\n_2=\emptyset}}&w(\n_1)w(\n_2)\mathbf
  I[\calS_g=S]\\
 &=\sum_{\substack{\partial
      \n_1^{\Lambda\cup\{g\}\setminus S}=\emptyset\\ 
      \partial\n_2=\emptyset}}w(\n_1^{\Lambda\cup\{g\}\setminus S})w(\n_2)\mathbf
  I[\calS_g=S]\Big(\sum_{\partial \n_1^S=\{0\}\Delta\{x\}}w(\n_1^S)\Big)\\
   &=\sum_{\substack{\partial
      \n_1^{\Lambda\cup\{g\}\setminus S}=\emptyset\\ 
      \partial\n_2=\emptyset}}w(\n_1^{\Lambda\cup\{g\}\setminus S})w(\n_2)\mathbf
  I[\calS_g=S]\Big(\sum_{\partial \n_1^S=\emptyset}w(\n_1^S)\Big)\langle\sigma_0\sigma_x\rangle_{S,\beta,0}\\
  &=\sum_{\substack{\partial
      \n_1=\emptyset\\ \partial\n_2=\emptyset}}w(\n_1)w(\n_2)\langle\sigma_0\sigma_x\rangle_{S,\beta,0}\,\mathbf
  I[\calS_g=S] .\end{align*} In the third line we used \eqref{eq:15} and in the
fourth line, we recombined $\n_1^S$ with $\n_1^{\Lambda\cup\{g\}\setminus S}$.
\end{proof}
Inequality~\eqref{eq:papa} and Claim 2 imply that for any $x,y\in\Lambda$,
  \begin{align*}
 \delta_{x,y}  &\ge\frac1{\langle \sigma_y \rangle_{\Lambda,\beta,h}}\, \sum_{\substack{S\subset\Lambda\\
    \text{s.t.~}0,x\in S\\ \text{and }y\in\Lambda\setminus S}}\ \sum_{\substack{
      \partial \n_1=\emptyset\\ \partial\n_2=\emptyset}}w(\n_1)w(\n_2)\langle\sigma_0\sigma_x\rangle_{S,\beta,0}\,\mathbf  I[\calS_g=S].
\end{align*}
By plugging
the inequality above in \eqref{eq:20}, we find
\begin{align}
   & \frac{d}{d\beta}\langle
  \sigma_0\rangle_{\Lambda,\beta,h}^2=2\langle \sigma_0\rangle_{\Lambda,\beta,h}\frac{d}{d\beta}\langle \sigma_0\rangle_{\Lambda,\beta,h}\\
  &\ge\frac2{Z^2}
 \sum_{\substack{S\subset\Lambda\\ S\ni 0}} \sum_{\substack{x\in S\\ y\in \Lambda\setminus S}}\frac{\langle \sigma_0\rangle_{\Lambda,\beta,h}}{\langle \sigma_y\rangle_{\Lambda,\beta,h}} \sum_{\substack{
      \partial \n_1=\emptyset\\ \partial\n_2=\emptyset}}w(\n_1)w(\n_2)J_{x,y}\langle\sigma_0\sigma_x\rangle_{S,\beta,0}\mathbf  I[\calS_g=S]\\
      &\ge\frac2{Z^2}
 \sum_{\substack{S\subset\Lambda\\ S\ni 0}}  \Big(\sum_{\substack{x\in S\\ y\in \Lambda\setminus S}}\frac{\langle \sigma_0\rangle_{\Lambda,\beta,h}}{\langle \sigma_y\rangle_{\Lambda,\beta,h}}J_{x,y}\langle\sigma_0\sigma_x\rangle_{S,\beta,0}\Big)\Big(\sum_{\substack{
      \partial \n_1=\emptyset\\ \partial\n_2=\emptyset}}w(\n_1)w(\n_2)\mathbf  I[\calS_g=S]\Big)\\
       &\ge\frac{2c(\Lambda,\beta,h)}{Z^2}
 \sum_{\substack{S\subset\Lambda\\ S\ni 0}}  \Big(\sum_{\substack{x\in S\\ y\in \Lambda\setminus S}}J_{x,y}\langle\sigma_0\sigma_x\rangle_{S,\beta,0}\Big)\Big(\sum_{\substack{
      \partial \n_1=\emptyset\\ \partial\n_2=\emptyset}}w(\n_1)w(\n_2)\mathbf  I[\calS_g=S]\Big) .\end{align}
Using that $J_{x,y}\ge \frac1{\beta}\tanh(\beta J_{x,y})$ gives that 
 $$\sum_{\substack{x\in S\\ y\in \Lambda\setminus S}}J_{x,y}\langle\sigma_0\sigma_x\rangle_{S,\beta,0}\ge\tfrac1\beta\varphi_\beta(S)-\sum_{\substack{x\in S\\ y\in V\setminus \Lambda}}J_{x,y}\langle\sigma_0\sigma_x\rangle_{S,\beta,0}.$$
We deduce that
 \begin{align}\nonumber
 \frac{d}{d\beta}\langle
   \sigma_0&\rangle_{\Lambda,\beta,h}^2\ge2c(\Lambda,\beta,h)\Big(\frac{1}{\beta}\cdot   \underbrace{\sum_{\substack{S\subset\Lambda \\ S\ni 0}}\varphi_\beta(S)\cdot \frac1{Z^2}\sum_{\substack{
      \partial \n_1=\emptyset\\ \partial\n_2=\emptyset}}w(\n_1)w(\n_2)\mathbf  I[\calS_g=S]}_{(A)}\\
   &-\underbrace{\sum_{\substack{S\subset\Lambda \\ S\ni 0}}\sum_{\substack{x\in S\\ y\in V\setminus \Lambda}}J_{x,y}\langle\sigma_0\sigma_x\rangle_{S,\beta,0}\frac1{Z^2}\sum_{\substack{
      \partial \n_1=\emptyset\\ \partial\n_2=\emptyset}}w(\n_1)w(\n_2)\mathbf  I[\calS_g=S]}_{(B)}\Big)\label{eq:21}\end{align}
      Taking the infimum over all the $\varphi_\beta(S)$ and then using \eqref{eq:16} and \eqref{eq:15} one more time, we obtain that
  \begin{align}   (A) &~\ge~\inf_{S\ni 0}\varphi_\beta(S)\,(1-\langle
  \sigma_0\rangle_{\Lambda,\beta,h}^2).\nonumber
\end{align}
Now, summing on $S$ after applying Claim 2 (backward compared to the last use of Claim 2) gives that
\begin{align}
(B)&=\sum_{\substack{x\in \Lambda\\ y\in V\setminus \Lambda}}J_{x,y}\frac1{Z^2}\sum_{\substack{
      \partial \n_1=\{0\}\Delta\{x\}\\ \partial\n_2=\emptyset}}w(\n_1)w(\n_2)(1-\mathbf  I[0\lr[\n_1+\n_2]g])\label{eq:0}\\
      &=\sum_{\substack{x\in \Lambda\\ y\in V\setminus\Lambda}}J_{x,y}(\langle\sigma_0\sigma_x\rangle_{\Lambda,\beta,h}-\langle\sigma_0\rangle_{\Lambda,\beta,h}\langle\sigma_x\rangle_{\Lambda,\beta,h}),
\end{align}
where, in the second line, we used \eqref{eq:16} and \eqref{eq:15} one last time. Plugging the expressions for (A) and (B) obtained above in \eqref{eq:21} implies the claim.
\end{proof}

\subsection{Proof of Items~\ref{item:5} and \ref{item:6}}
\label{sec:proof-items-5-6}

In this section, we show that Items~\ref{item:5} and \ref{item:6} in
Theorem~\ref{thm:Ising} hold with $\tilde{\beta_c}$ in place of
$\beta_c$.

We need a replacement for the BK inequality used in the case of Bernoulli
percolation. The relevant tool for the Ising model will be a modified version of
Simon's inequality. The original inequality can be found in \cite{Sim80}, see
also \cite{Lie80} for an improvement. (Those previous versions do not suffice
for our application).
\begin{lemma}[Modified Simon's inequality]\label{lem:finiteCriterion:ising}
  Let $S$ be a finite subset of $V$ containing 0. For every $z\in V\setminus S$,
$$\langle \sigma_0\sigma_z\rangle_{\beta}^+\le \sum_{x\in S}\sum_{y\notin S}\tanh(\beta J_{x,y})\langle \sigma_0\sigma_x\rangle_{S,\beta,0}\langle \sigma_y\sigma_z\rangle_{\beta}^+.$$
\end{lemma}

\begin{proof}
Fix $h\ge0$ and $\Lambda$ a finite subset of $V$ containing $S$. We introduce the ghost vertex $g$ as before.

We consider the backbone representation of the Ising model on $\Lambda\cup\{g\}$ defined in the previous section. Let $\omega=(v_k)_{0\le k\le K}$ be a backbone from $0$ to $z$ (it may go through $g$). Since $z\notin S$, one can define the first $k$ such that $v_{k}\in \Lambda\setminus S$ and set $y=v_k$. Also set $x$ to be the vertex of $S$ visited last by the backbone before reaching $y$. The following occurs:
\begin{itemize}[noitemsep,nolistsep]\item $\omega$ goes from $0$ to $x$ staying in $S\cup\{g\}$,
\item then $\omega$ goes from $x$ to $y$ either in one step by using the edge $\{x,y\}$ or in two steps by going through $\{x,g\}$ and then $\{g,y\}$,
\item finally $\omega$ goes from $y$ to $z$ in $\Lambda\cup\{g\}$.
\end{itemize}
Call $\omega_1$ the part of the walk $\omega$ from $0$ to $x$, $\omega_2$ the walk from $x$ to $y$, and $\omega_3$ the reminder of the walk $\omega$. 

Using Property {\bf P1} of the backbone representation, we can write
\begin{align*}\langle \sigma_0\sigma_z\rangle_{\Lambda,\beta,h}&=\sum_{\partial\omega=\{0,z\}}\rho_{\Lambda}(\omega).\end{align*}
Then, {\bf P2} applied with $\omega_1$ and $\omega_2\circ\omega_3$ and then with $\omega_2$ and $\omega_3$ implies that $\langle \sigma_0\sigma_z\rangle_{\Lambda,\beta,h}$ is bounded from above by
\begin{align*}
&\sum_{x\in S}\sum_{y\in \Lambda\setminus S}\sum_{\partial\omega_1=\{0,x\}}\rho_\Lambda(\omega_1)\Big(\sum_{\partial\omega_2=\{x,y\}}\rho_{\Lambda\setminus\overline{\omega_1}}(\omega_2)\Big(\sum_{\partial\omega_3=\{y,z\}}\rho_{\Lambda\setminus\overline{\omega_1\circ\omega_2}}(\omega_3)\Big)\Big).\end{align*}
{\bf P1} and then Griffiths' inequality \eqref{eq:14} imply that 
$$\sum_{\partial\omega_3=\{y,z\}}\rho_{\Lambda\setminus\overline{\omega_1\circ\omega_2}}(\omega_3)=\langle \sigma_y\sigma_z\rangle_{\Lambda\setminus \overline{\omega_1\circ\omega_2},\beta,h}\le \langle \sigma_y\sigma_z\rangle_{\Lambda,\beta,h}.$$ Inserting this in the last displayed equation gives
 \begin{align*}
\langle \sigma_0\sigma_z\rangle_{\Lambda,\beta,h}&\le  \sum_{x\in S}\sum_{y\in  \Lambda\setminus S}\Big(\sum_{\partial\omega_1=\{0,x\}}\rho_\Lambda(\omega_1)\Big(\sum_{\partial\omega_2=\{x,y\}}\rho_{\Lambda\setminus\overline{\omega_1}}(\omega_2)\Big)\Big)\langle \sigma_y\sigma_z\rangle_{\Lambda,\beta,h}.\end{align*}
Since $\omega_2$ uses only vertices $x$, $y$ and $g$, {\bf P3} and then {\bf P1} lead to
$$\sum_{\partial\omega_2=\{x,y\}}\rho_{\Lambda\setminus\overline{\omega_1}}(\omega_2)\le \sum_{\partial\omega_2=\{x,y\}}\rho_{\{x,y\}}(\omega_2)=\langle \sigma_x\sigma_y\rangle_{\{x,y\},\beta,h}$$
which gives
\begin{align*}
\langle \sigma_0\sigma_z\rangle_{\Lambda,\beta,h}&\le  \sum_{x\in S}\sum_{y\in  \Lambda\setminus S}\Big(\sum_{\partial\omega_1=\{0,x\}}\rho_\Lambda(\omega_1)\Big)\langle \sigma_x\sigma_y\rangle_{\{x,y\},\beta,h}\langle \sigma_y\sigma_z\rangle_{\Lambda,\beta,h}.\end{align*}
Finally, {\bf P3} can be used with the fact that $\omega_1\subset S\cup\{g\}$ to show that
\begin{align*}
\langle \sigma_0\sigma_z\rangle_{\Lambda,\beta,h}&\le  \sum_{x\in S}\sum_{y\in  \Lambda\setminus S}\Big(\sum_{\partial\omega_1=\{0,x\}}\rho_S(\omega_1)\Big)\langle \sigma_x\sigma_y\rangle_{\{x,y\},\beta,h}\langle \sigma_y\sigma_z\rangle_{\Lambda,\beta,h}\\
&\le  \sum_{x\in S}\sum_{y\in  \Lambda\setminus S}\langle \sigma_0\sigma_x\rangle_{S,\beta,h}\,\langle \sigma_x\sigma_y\rangle_{\{x,y\},\beta,h}\,\langle \sigma_y\sigma_z\rangle_{\Lambda,\beta,h}
\end{align*}
(we used {\bf P1} in the second line). Let $\Lambda$ tend to $V$ to obtain 
\begin{align*}
\langle \sigma_0\sigma_z\rangle_{\beta,h}&\le   \sum_{x\in S}\sum_{y\in  V\setminus S}\langle \sigma_0\sigma_x\rangle_{S,\beta,h}\,\langle \sigma_x\sigma_y\rangle_{\{x,y\},\beta,h}\,\langle \sigma_y\sigma_z\rangle_{\beta,h}.
\end{align*}
Let now $h$ tend to 0 to find
\begin{itemize}[noitemsep]
\item $\langle \sigma_0\sigma_z\rangle_{\beta,h}$ and $\langle \sigma_y\sigma_z\rangle_{\beta,h}$ tend to $\langle \sigma_0\sigma_z\rangle_{\beta}^+$ and $\langle \sigma_y\sigma_z\rangle_{\beta}^+$ respectively.
\item $\langle \sigma_0\sigma_x\rangle_{S,\beta,h}$ tends to $\langle \sigma_0\sigma_x\rangle_{S,\beta,0}$ (since $S$ is finite).
\item  $\langle \sigma_x\sigma_y\rangle_{\{x,y\},\beta,h}$ tends to $\langle \sigma_x\sigma_y\rangle_{\{x,y\},\beta,0}=\tanh(\beta J_{x,y})$.
\end{itemize}
Using one last time that $S$ is finite, we deduce that 
$$\langle \sigma_0\sigma_z\rangle_{\beta}^+\le \sum_{x\in S}\sum_{y\in V\setminus S}\tanh(\beta J_{x,y})\langle \sigma_0\sigma_x\rangle_{S,\beta,0}\langle \sigma_y\sigma_z\rangle_{\beta}^+.$$
\end{proof}

We are now in a position to conclude the proof. 
Let $\beta<\tilde{\beta_c}$. Fix a finite set $S$ such that
$\varphi_\beta(S)<1$. Define, $$\chi_n(\beta):=\sup\big\{\sum_{z\in\Lambda}\langle \sigma_0 \sigma_z\rangle_{\beta}^+:\Lambda\subset V\text{ with }|\Lambda|\le n\big\}.$$
By the same reasoning as for percolation, 
Lemma~\ref{lem:finiteCriterion:ising} shows that
\begin{equation}
  \label{eq:22}
\sum_{z\in\Lambda}\langle \sigma_0 \sigma_z\rangle_\beta^+\le |S|+\sum_{x\in
  S}\sum_{y\in V\setminus S}\tanh(\beta J_{x,y})\langle
\sigma_0\sigma_x\rangle_{S,\beta,0}\Big(\sum_{z\in \Lambda\setminus S}\langle \sigma_y \sigma_z\rangle_{\beta}^+\Big).
\end{equation}
Using the invariance under translations and taking the supremum over sets $\Lambda$ of volume $n$, we immediately get that
$ \chi_n(\beta)<|S|/[1-\varphi_\beta(S)]$ uniformly in $n$. Letting $n$ tend to infinity $\infty$ gives the second item.
\bigbreak We finish by the proof of the third item. Let $R$ be
the range of the $(J_{x,y})_{x,y\in V}$, and let $L$ be such that $S\subset
\Lambda_{L-R}$. Lemma~\ref{lem:finiteCriterion:ising} implies that for any $z$ with $\mathrm{d}(0,z)\ge n>L$, 
\begin{align*}\langle \sigma_0\sigma_z\rangle_{\beta}^+&\le \sum_{x\in S}\sum_{y\in V\setminus S}\tanh(\beta J_{x,y})\langle \sigma_0\sigma_x\rangle_{S,\beta,0}\langle \sigma_y\sigma_z\rangle_{\beta}^+\le \varphi_\beta(S)\,\max_{y\in\Lambda_L}\langle \sigma_y\sigma_z\rangle_{\beta}^+.\end{align*}
Note that $d(y,z)\ge n-L$. If $d(y,z)\le L$, we bound $\langle \sigma_y\sigma_z\rangle_\beta^+$ by 1, while if $d(y,z)>L$, we apply the previous inequality to $y$ and $z$ instead of $0$ and $z$. The proof follows by iterating $\lfloor n/L\rfloor$ times this strategy.
\bibliographystyle{alpha}

\subsection{Proof of Proposition~\ref{prop:field2}}\label{sec:901}

Let us introduce $M(\beta,h)=\langle\sigma_0\rangle_{\beta,h}$.  Recall that
$M(\beta,h)$ is differentiable in $(\beta,h)$ away from the line $h=0$.

As in the case of percolation, the proof in \cite{AizBarFer87} invokes three
inequalities (the pages below refer to the numbering in~\cite{AizBarFer87}): the differential inequality (1.12) page 348,
\begin{equation}\label{eq:his}
  \frac{\partial M}{\partial\beta}\le \big(\sum_{y\in V}J_{0,y}\big)\, M\,\frac{\partial M}{\partial
    h},
\end{equation}
the more difficult differential inequality (1.9) page 347, as well as (1.13) page 348.
Below, we combine Lemma~\ref{lem:ising:meanField} with \eqref{eq:his} to conclude the proof without using (1.9) or (1.13) of \cite{AizBarFer87}.

Since $M(\beta,h)$ is differentiable for $h>0$, we may pass to the limit $\Lambda\nearrow V$ in Lemma~\ref{lem:ising:meanField} to get
$$M\,\frac{\partial M}{\partial \beta} \ge \frac{2}{\beta}\,(1-M^2)$$
for $h>0$ and $\beta\ge \beta_c$ (once again we used that $\varphi_\beta(S)\ge1$ for any finite $S\ni0$ and for any $\beta\ge\beta_c$, see the comment before Proposition~\ref{prop:b}).
Together with \eqref{eq:his}, we find
\begin{equation*}
\frac{2}{\beta} \,(1-M^2) \le M\,\frac{\partial M}{\partial\beta}\le \big(\sum_{y\in V}J_{0,y}\big)\, M^2\,\frac{\partial M}{\partial
    h}
\end{equation*}
which immediately implies that there exists a
constant $c>0$ such that for any $h>0$, 
$$\langle\sigma_0\rangle_{\beta_c,h}=M(\beta_c,h)~\ge~ ch^{1/3}.$$

To conclude this article, let us recall the proof of \eqref{eq:his} for completeness.
\begin{lemma}[(1.12) page 348 of \cite{AizBarFer87}] On $(0,1)\times(0,\infty)$,
  the function $M$ satisfies the following differential inequality:
\begin{equation*}
  \frac{\partial M}{\partial\beta}\le \big(\sum_{y\in V}J_{0,y}\big)\, M\,\frac{\partial M}{\partial
    h}.
\end{equation*}
\end{lemma}
\begin{proof}
Let $\beta>0$, $h>0$. We have 
\begin{equation*}
  \frac{\partial M}{\partial\beta}=\sum_{\{x,y\}}J_{x,y}\big(\langle
  \sigma_0\sigma_x\sigma_y \rangle_{\beta,h} -\langle
  \sigma_0 \rangle_{\beta,h} \langle
  \sigma_x\sigma_y \rangle_{\beta,h}\big).
\end{equation*}
The Griffith-Hurst-Sherman inequality \cite[(2.8)]{GriHurShe70} gives\begin{align*}\langle
  \sigma_0\sigma_x\sigma_y \rangle_{\beta,h} -\langle
  \sigma_0 \rangle_{\beta,h} \langle
  \sigma_x\sigma_y \rangle_{\beta,h}&\le (\langle
  \sigma_0\sigma_x \rangle_{\beta,h} -\langle
  \sigma_0 \rangle_{\beta,h} \langle
  \sigma_x \rangle_{\beta,h})\langle
  \sigma_y \rangle_{\beta,h}\\
  &\ \ +(\langle
  \sigma_0\sigma_y \rangle_{\beta,h} -\langle
  \sigma_0 \rangle_{\beta,h} \langle
  \sigma_y \rangle_{\beta,h})\langle
  \sigma_x \rangle_{\beta,h}.\end{align*}
  This implies that
  $$\frac{\partial M}{\partial\beta}\le \Big(\sum_{y\in V}J_{0,y}\Big)\,M\,\Big(\sum_{x}\langle
  \sigma_0\sigma_x \rangle_{\beta,h} -\langle
  \sigma_0 \rangle_{\beta,h} \langle
  \sigma_x \rangle_{\beta,h}\Big)=\Big(\sum_{y\in V}J_{0,y}\Big)\,M\,\frac{\partial M}{\partial h}.$$
\end{proof}

\paragraph{Acknowledgments} This work was supported by a grant from the Swiss NSF and the NCCR SwissMap also funded by the Swiss NSF. We thank M. Aizenman and G. Grimmett for useful comments on this paper. We also thank D. Ioffe, A. Glazman and M. Lis for pointing out mistakes in previous versions of the manuscript. Finally, we thank the anonymous referees for numerous important comments and suggestions.
\bibliography{biblicomplete}

\begin{thebibliography}{BNP11}

\bibitem[AB87]{AizBar87}
M.~Aizenman and D.~J. Barsky.
\newblock Sharpness of the phase transition in percolation models.
\newblock {\em Comm. Math. Phys.}, 108(3):489--526, 1987.

\bibitem[ABF87]{AizBarFer87}
M.~Aizenman, D.~J. Barsky, and R.~Fern{{\'a}}ndez.
\newblock The phase \mbox{transition} in a general class of {I}sing-type models
  is sharp.
\newblock {\em J. Statist. Phys.}, 47(3-4):343--374, 1987.

\bibitem[AF86]{AizFer86}
M.~Aizenman and R.~Fern{{\'a}}ndez.
\newblock On the critical behavior of the magnetization in high-dimensional
  {I}sing models.
\newblock {\em J. Statist. Phys.}, 44(3-4):393--454, 1986.

\bibitem[Aiz82]{Aiz82}
M.~Aizenman.
\newblock Geometric analysis of {$\varphi ^{4}$} fields and {I}sing models.
  {I}, {II}.
\newblock {\em Comm. Math. Phys.}, 86(1):1--48, 1982.

\bibitem[AN84]{AizNew84}
M.~Aizenman and C.~M. Newman.
\newblock Tree graph inequalities and critical behavior in percolation models.
\newblock {\em Journal of Statistical Physics}, 36(1-2):107--143, 1984.

\bibitem[AV08]{antunovic2008sharpness}
T.~Antunovi{\'c} and I.~Veseli{\'c}.
\newblock Sharpness of the phase transition and exponential decay of the
  subcritical cluster size for percolation on quasi-transitive graphs.
\newblock {\em Journal of Statistical Physics}, 130(5):983--1009, 2008.

\bibitem[BD12a]{BefDum12}
V.~Beffara and H.~{Duminil-Copin}.
\newblock The self-dual point of the two-dimensional random-cluster model is
  critical for {$q\geq 1$}.
\newblock {\em Probab. Theory Related Fields}, 153(3-4):511--542, 2012.

\bibitem[BD12b]{BefDum12a}
V.~Beffara and H.~{Duminil-Copin}.
\newblock Smirnov's fermionic observable away from criticality.
\newblock {\em Ann. Probab.}, 40(6):2667--2689, 2012.

\bibitem[BNP11]{BenNacPer11}
I.~Benjamini, A.~Nachmias, and Y.~Peres.
\newblock Is the critical \mbox{percolation} probability local?
\newblock {\em Probab. Theory Related Fields}, 149(1-2):261--269, 2011.

\bibitem[BR06]{BolRio06c}
B{{\'e}}la Bollob{{\'a}}s and Oliver Riordan.
\newblock A short proof of the {H}arris-{K}esten theorem.
\newblock {\em Bull. London Math. Soc.}, 38(3):470--484, 2006.

\bibitem[CC87]{ChaCha87}
J.~T. Chayes and L.~Chayes.
\newblock The mean field bound for the order parameter of {B}ernoulli
  percolation.
\newblock In {\em Percolation theory and ergodic theory of infinite particle
  systems ({M}inneapolis, {M}inn., 1984--1985)}, volume~8 of {\em IMA Vol.
  Math. Appl.}, pages 49--71. Springer, New York, 1987.

\bibitem[DST15]{DumSidTas13}
H.~{Duminil-Copin}, V.~Sidoravicius, and V.~Tassion.
\newblock Continuity of the phase transition for planar random-cluster and
  {P}otts models with $1\le q\le 4$.
\newblock arXiv:1505.04159, 2015.

\bibitem[DT15]{DumTas15}
H.~{Duminil-Copin} and V.~Tassion.
\newblock A new proof of the sharpness of the phase transition for {B}ernoulli
  percolation and the {I}sing model.
\newblock arXiv:1502.03050, 2015.

\bibitem[GHS70]{GriHurShe70}
Robert~B. Griffiths, C.~A. Hurst, and S.~Sherman.
\newblock Concavity of magnetization of an {I}sing ferromagnet in a positive
  external field.
\newblock {\em J. Mathematical Phys.}, 11:790--795, 1970.

\bibitem[Gri67]{Gri67}
R.~B. Griffiths.
\newblock Correlation in {I}sing ferromagnets {I}, {II}.
\newblock {\em J. Math. Phys.}, 8:478--489, 1967.

\bibitem[Gri99]{Gri99a}
G.~Grimmett.
\newblock {\em Percolation}, volume 321 of {\em Grundlehren der Mathematischen
  Wissenschaften [Fundamental Principles of \mbox{Mathematical} Sciences]}.
\newblock Springer-Verlag, Berlin, second edition, 1999.

\bibitem[Gri06]{Gri06}
G.~Grimmett.
\newblock {\em The random-cluster model}, volume 333 of {\em Grundlehren der
  Mathematischen Wissenschaften [Fundamental Principles of Mathematical
  Sciences]}.
\newblock Springer-Verlag, Berlin, 2006.

\bibitem[Ham57]{Ham57}
J.~M. Hammersley.
\newblock Percolation processes: {L}ower bounds for the critical probability.
\newblock {\em Ann. Math. Statist.}, 28:790--795, 1957.

\bibitem[Har60]{Har60}
T.~E. Harris.
\newblock A lower bound for the critical probability in a certain percolation
  process.
\newblock {\em Proc. Cambridge Philos. Soc.}, 56:13--20, 1960.

\bibitem[Kes80]{Kes80}
H.~Kesten.
\newblock The critical probability of bond percolation on the square lattice
  equals {${1\over 2}$}.
\newblock {\em Comm. Math. Phys.}, 74(1):41--59, 1980.

\bibitem[Lie80]{Lie80}
E.~H. Lieb.
\newblock A refinement of {S}imon's correlation inequality.
\newblock {\em Comm. Math. Phys.}, 77(2):127--135, 1980.

\bibitem[Men86]{Men86}
M.~V. Menshikov.
\newblock Coincidence of critical points in percolation problems.
\newblock {\em Dokl. Akad. Nauk SSSR}, 288(6):1308--1311, 1986.

\bibitem[Ons44]{Ons44}
L.~Onsager.
\newblock Crystal statistics. {I}. {A} two-dimensional model with an
  order-disorder transition.
\newblock {\em Phys. Rev. (2)}, 65:117--149, 1944.

\bibitem[Rus78]{Rus78}
L.~Russo.
\newblock A note on percolation.
\newblock {\em Z. Wahrscheinlichkeitstheorie und Verw. Gebiete}, 43(1):39--48,
  1978.

\bibitem[Sim80]{Sim80}
B.~Simon.
\newblock Correlation inequalities and the decay of correlations in
  ferromagnets.
\newblock {\em Comm. Math. Phys.}, 77(2):111--126, 1980.

\end{thebibliography}
\small\begin{flushright}
\textsc{D\'epartement de Math\'ematiques}
  \textsc{Universit\'e de Gen\`eve}
  \textsc{Gen\`eve, Switzerland}
  \textsc{E-mail:} \texttt{hugo.duminil@unige.ch}, \texttt{vincent.tassion@unige.ch}
\end{flushright}
\end{document}